\documentclass[a4paper,12pt]{article}
\usepackage{amsmath}
\usepackage{amssymb}

\usepackage{setspace}
\usepackage{fullpage}
\usepackage{bbm}
\usepackage{amsthm}
\newtheorem{theorem}{Theorem}[section]
\newtheorem{lemma}[theorem]{Lemma}
\newtheorem*{rtheorem}{Main Theorem}
\newtheorem{corollary}[theorem]{Corollary}

\theoremstyle{definition}

\newtheorem*{definition}{Definition}

\newtheorem*{remark}{Remark}

\usepackage{pdfsync}
\def\PG{\mathrm{PG}} \def\AG{\mathrm{AG}} 
  \def\Persp{\mathrm{Persp}}

\def\Aut{\mathrm{Aut}}
\def\PGammaL{\mathrm{P}\Gamma\mathrm{L}}
\def\PGL{\mathrm{PGL}}

\def\A{\mathcal{A}}  
\def\D{\mathcal{D}}  
 \def\K{\mathcal{K}}
\def\L{\mathcal{L}}  
 \def\P{\mathcal{P}}
\def\R{\mathcal{R}}

\def\F{\mathbb{F}}

\def\a{\alpha}

\title{The isomorphism problem for linear representations and their graphs}
\author{Philippe Cara\thanks{Partially supported by grant 15.263.08 of the `Fonds Wetenschappelijk Onderzoek-Vlaanderen'.} \and Sara Rottey \and Geertrui Van de Voorde\thanks{Supported by VUB-grant GOA62.}}
\date{}
\begin{document}
\maketitle
\begin{abstract} In this paper, we study the isomorphism problem for {\em linear representations}. A linear representation $T_n^*(\K)$ of a point set $\K$ is a point-line geometry, embedded in a projective space $\PG(n+1,q)$, where $\K$ is contained in a hyperplane. We put constraints on $\K$ which ensure that every automorphism of $T_n^*(\K)$ is induced by a collineation of the ambient projective space. This allows us to show that, under certain conditions, two linear representations $T_n^*(\K)$ and $T_n^*(\K')$ are isomorphic if and only if the point sets $\K$ and $\K'$ are $\PGammaL$-equivalent. We also deal with the slightly more general problem of isomorphic {\em incidence graphs} of linear representations.

In the last part of this paper, we give an explicit description of the group of automorphisms of $T_n^*(\K)$ that are induced by collineations of $\PG(n+1,q)$.

\end{abstract}


\section{Introduction}\label{section1}

A natural question in finite geometry is the following.

\begin{itemize}
\item[(Q)]
Let $G_1$ and $G_2$ be two geometries embedded in a projective space. Is every isomorphism between $G_1$ and $G_2$ induced by a collineation of the ambient projective space?
\end{itemize}

Of course, the answer to this question will strongly depend on the properties of the geometries $G_1$ and $G_2$. In this paper, we study this problem for particular geometries, namely {\em linear representations}.

\begin{definition} Let $\K$ be a point set in $H_\infty=\PG(n,q)$ and embed $H_\infty$ in $\PG(n+1,q)$. The {\em linear representation} $T_n^*(\K)$ of $\K$ is a point-line incidence structure with natural incidence, point set $\P$ and line set $\L$ as follows:
\begin{itemize}
\item[$\P$:] {\em affine} points of $\PG(n+1,q)$  (i.e. the points of $\PG(n+1,q)\setminus H_\infty$),
\item[$\L$:] lines of $\PG(n+1,q)$ through a point of $\K$, but not lying in $H_\infty$.
\end{itemize}
\end{definition}

We see that a linear representation $T_n^*(\K)$ in $\PG(n+1,q)$ is entirely defined by the point set $\K$ at infinity.

If the answer to (Q) is affirmative, then it follows that two linear representations $T_n^*(\K)$ and $T_n^*(\K')$ are isomorphic if and only if the point sets $\K$ and $\K'$ are $\PGammaL$-equivalent, which we denote by $\K\cong \K'$. In Section \ref{sec3}, we provide certain (weak) conditions on the point sets $\K$ and $\K'$ such that the answer to question (Q) is affirmative for $T_n^*(\K)$ and $T_n^*(\K')$. This leads to the following theorem (Corollary \ref{hoofd}).

\begin{rtheorem} Let $q>2$. Let $\K$ and $\K'$ denote point sets in $H_\infty=\PG(n,q)$ such that
\begin{itemize}
\item there is no plane of $H_\infty$ intersecting $\K$ in two intersecting lines, or in two intersecting lines minus their intersection point;
\item the closure $\overline{\K}$ (as defined in Section~\ref{closure}) is equal to $H_\infty$.
\end{itemize}
If $\alpha$ is an isomorphism of incidence structures between $T_n^*(\K)$ and $T_n^*(\K')$, then $\alpha$ is induced by a collineation of the ambient space mapping $\K$ to $\K'$.
\end{rtheorem}

For such $\K$, the automorphism group of $T_n^*(\K)$ is studied in Section \ref{sec5}, where an explicit description is given. In Section \ref{sec4}, we investigate what happens when the conditions on $\K$, found in Section \ref{sec3}, are not fulfilled.

Linear representations are mostly studied for point sets $\K$ that possess a lot of symmetry. For example, in the case $n=2$ and $\K$ a hyperoval, $T_2^*(\K)$ is a generalised quadrangle of order $(q-1,q+1)$.  Bichara, Mazzocca and Somma showed in \cite{bichara} that  for $\K, \K'$ hyperovals, $T_2^*(\K)\cong T_2^*(\K')$ if and only if $\K\cong \K'$. The answer to (Q) is proven positive when $\K$ is a regular hyperoval in $\PG(2,q)$ in \cite{Grund}. When $\K$ is a Buekenhout-Metz unital, question (Q) is answered by De Winter in \cite{stefaan}; in this case, the linear representation $T_2^*(\K)$ is a {\em semipartial geometry}. It is worth noticing that our result includes these cases.

A further generalisation of the concept of semipartial geometries is that of $(\alpha,\beta)$-geometries. Also these geometries can be constructed using linear representations. For more information, we refer to \cite{declerck}.

We will also deal with a slight variation of the question (Q), namely, we will consider isomorphisms between the {\em incidence graphs of linear representations}.

\begin{definition} We denote the point-line incidence graph of $T_n^*(\K)$ by $\Gamma_{n,q}(\K)$, i.e. the bipartite graph with classes $\P$ and $\L$ and adjacency corresponding to the natural incidence of the geometry.
\end{definition}
Whenever we consider the incidence graph $\Gamma_{n,q}(\K)$ of some linear representation $T_n^*(\K)$ of $\K$, we still regard the set of vertices as a set of points and lines in $\PG(n+1,q)$. In this way we can use the inherited properties of this space and borrow expressions such as the span of points, a subspace, incidence, and others.

It is easy to see that dealing with question (Q) for  $\Gamma_{n,q}(\K)$ and  $\Gamma_{n,q}(\K')$ is essentially the same as dealing with the question for $T_n^*(\K)$ and $T_n^*(\K')$ when there is no isomorphism between $\Gamma_{n,q}(\K)$ and $\Gamma_{n,q}(\K')$ mapping a vertex of $\Gamma_{n,q}(\K)$ corresponding to a point onto a vertex of $\Gamma_{n,q}(\K')$ corresponding to a line. In Section \ref{sec2}, we give a condition on $\K,\K'$ to ensure that every isomorphism between $\Gamma_{n,q}(\K)$ and $\Gamma_{n,q}(\K')$ preserves the set $\P$.

One of the reasons to generalise the answer to (Q) for incidence graphs of linear representations is the paper \cite{wij}, where we use them to construct new mutually non-isomorphic infinite families of semisymmetric graphs, i.e. regular edge-transitive, but not vertex-transitive graphs.

Also from a computational point of view, this generalisation is worthwhile: it follows from our main theorem that, under certain conditions, $\K\cong \K'$ if and only if $\Gamma_{n,q}(\K)\cong \Gamma_{n,q}(\K')$. By computer, checking whether two graphs are isomorphic often can be done a lot faster than checking whether two point sets in a projective space are $\PGammaL$-equivalent.

\section{A property of $\Gamma_{n,q}(\K)$}\label{sec2}

Note that an automorphism of $T_n^*(\K)$ as an incidence structure maps always points onto points and lines onto lines, whereas an automorphism of $\Gamma_{n,q}(\K)$ might map vertices corresponding to points onto vertices corresponding to lines. Of course, in this latter case, the sets $\P$ and $\L$ have equal size, hence $|\K|=q$.

Consider two linear representations $T_n^*(\K)$ and $T_n^*(\K')$. We may assume that both geometries are embedded in the same $(n+1)$-dimensional projective space, and that $\K$ and $\K'$ are embedded in the same hyperplane $H_\infty$. This implies that, for both bipartite graphs $\Gamma_{n,q}(\K)$ and $\Gamma_{n,q}(\K')$ we can denote the set of vertices of the first type (the points) by $\P$. We denote the set of vertices of the second type (the lines) in $\Gamma_{n,q}(\K)$ by $\L$ and in $\Gamma_{n,q}(\K')$ by $\L'$. For a vertex $v$ in a graph $\Gamma$ and a positive integer $i$ we write
$\Gamma_{i}(v)$ for the set of vertices at distance $i$ from $v$.

The following lemma provides a condition which forces the neighbourhood of a vertex in the set $\P$ and of a vertex in the set $\L$ to be essentially different.

\begin{lemma}\label{NVT}
Let $\K$ be a set of points of $H_\infty$ such that every point of $H_\infty \backslash \K$ lies on at least one tangent line to $\K$, then $\forall P \in \P, \forall L \in \L: \Gamma_{n,q}(\K)_{4}(P) \not\cong \Gamma_{n,q}(\K)_{4}(L)$.
\end{lemma}
\begin{proof}
We will prove that, for every line $L \in \L$, the set of vertices $\Gamma_{n,q}(\K)_{4}(L)$
contains at least one vertex that has all its neighbours in
$\Gamma_{n,q}(\K)_{3}(L)$, while for every point $P \in \P$, a vertex in the set $\Gamma_{n,q}(\K)_{4}(P)$
cannot have all its neighbours in $\Gamma_{n,q}(\K)_{3}(P)$.

To prove the first claim, consider a line $L \in \L$ with $L \cap H_\infty= P_1 \in \K$. Choose an affine point $Q$ on $L$ and a point $P_2 \in \K$ different from $P_1$.
Take a point $R$ on $Q P_2$, not equal to $Q$ or $P_2$, then clearly the line $R P_1 \in \Gamma_{n,q}(\K)_{4}(L)$. We will show that $R P_1$ has all its neighbours in $\Gamma_{n,q}(\K)_{3}(L)$. Consider a neighbour $S$ of $R P_1$, i.e $ S \in R P_1 \setminus \{P_1\}$.
The line $S P_2$ meets $L$ in a point $T$. Since $T \in \Gamma_{n,q}(\K)_{1}(L)$ and $T P_2 \in \Gamma_{n,q}(\K)_{2}(L)$, it follows that $S \in \Gamma_{n,q}(\K)_{3}(L)$.

Consider now a point $P \in \P$ and a point $T \in \Gamma_{n,q}(\K)_{4}(P)$. Consider the following minimal path of length 4 from $T$ to $P$: the point $T$, a line $Q_1 P_1 \in \Gamma_{n,q}(\K)_{3}(P)$ containing $T$ for some $P_1 \in \K$ and affine point $Q_1 \in \Gamma_{n,q}(\K)_{2}(P)$, the line $P P_2 \in \Gamma_{n,q}(\K)_{1}(P)$ containing $Q_1$, for some $P_2 \in \K$ different from $P_1$, and finally the point $P$. Consider the point $R= PT \cap H_\infty$, then it follows from our construction that $R$ lies on the line $P_1P_2$. Moreover $R$ is different from $P_1,P_2$, and hence, since $PR \notin \Gamma_{n,q}(\K)_{1}(P)$, we have $R$ not in $\K$. By assumption, there is a tangent line of $\K$ through $R$, say $RP_3$, with $P_3\in \K$. The line $TP_3$ is a neighbour of $T$. Suppose that $TP_3$ belongs to $\Gamma_{n,q}(\K)_{3}(P)$, then there exists a line $PT'$ through a point $P_4\in \K$, with $T'$ on $TP_3$, which implies that $RP_3$ contains the point $P_4\in \K$, a contradiction.
\end{proof}

\begin{corollary}\label{cor22}
Let $|\K|\neq q$ or let $\K$ be a set of points of $H_\infty$ such that every point of $H_\infty \backslash \K$ lies on at least one tangent line to $\K$. Suppose $\alpha$ is an isomorphism between $\Gamma_{n,q}(\K)$ and $\Gamma_{n,q}(\K')$, for some set $\K'$ in $H_\infty$, then $\alpha$ stabilises $\P$.
\end{corollary}
\begin{proof}
Since any graph isomorphism preserves distance and hence neighbourhoods, no isomorphism between $\Gamma_{n,q}(\K)$ and $\Gamma_{n,q}(\K')$ maps a vertex in $\P$ to a vertex in $\L'$.
\end{proof}

\begin{remark}\label{rem1} If $\K$ does not satisfy the conditions of Corollary \ref{cor22}, then there exist counterexamples to the corollary. Let $\K$ be the $q$-arc $\{(0,1,x,x^2)\mid x \in \F_q\}$, $q$ even, embedded in the plane $H_\infty$ of $\PG(3,q)$ with equation $X_0=0$. Consider the mapping $\phi$ from the affine point $(1,a,b,c)$ to the line $\langle (0,1,a,a^2),(1,0,c,b^2)\rangle$, this map preserves the edges of the graph $\Gamma_{2,q}(\K)$ but switches the sets $\P$ and $\L$.
\end{remark}

\section{The isomorphism problem for linear representations}\label{sec3}
In this section, we will deal with question (Q) for (the incidence graphs of) linear representations. An isomorphism between $\Gamma_{n,q}(\K)$ and $\Gamma_{n,q}(\K')$ (or $T_n^*(\K)$ and $T_n^*(\K')$) that is induced by a collineation of the ambient projective space $\PG(n+1,q)$ is called {\em geometric}.

As said before, we will provide certain conditions on $\K$, to ensure that every isomorphism between $\Gamma_{n,q}(\K)$ and $\Gamma_{n,q}(\K')$ (or $T_n^*(\K)$ and $T_n^*(\K')$) is geometric. One of these conditions is related to the {\em closure} of $\K$, which we will now define.

\begin{definition}\label{closure}
If a point set $S$ contains a frame of $\PG(n,q)$, then its {\em closure} $\overline{S}$ consists of the points of the
smallest $n$-dimensional subgeometry of $\PG(n,q)$ containing all points of $S$.
\end{definition}
The closure $\overline{S}$ of a point set $S$ can be constructed recursively as follows:
\begin{itemize}
\item[(i)] determine the set $\A$ of all subspaces of $\PG(n,q)$ spanned by an arbitrary number of points of $S$;
\item[(ii)] determine the set $\overline{S}$ of points that occur as the exact intersection of
two subspaces in $\A$, if $\overline{S} \neq S$ replace $S$ by
$\overline{S}$ and go to~(i), otherwise stop.
\end{itemize}

For $n=2$, this recursive construction coincides with the definition of the closure of a set of points in a plane containing a quadrangle, given in \cite[Chapter XI]{Hughes}.

We will also introduce the notion of {\em rigid} subspaces:

\begin{definition} Let $\alpha$ be an isomorphism between the graphs $\Gamma_{n,q}(\K)$ and $\Gamma_{n,q}(\K')$ (or between $T_n^*(\K)$ and $T_n^*(\K')$) that preserves the set $\P$.  We
will say that a $k$-subspace $\pi_\infty$ of $H_\infty$ is \emph{$\a$-rigid}
if for every $(k+1)$-subspace $\pi$ through
$\pi_\infty$, not contained in $H_\infty$, the point set $\{ \a(P)|P
\in \pi, P\notin H_\infty\}$ spans a $(k+1)$-subspace.
\end{definition}

We see that the definition of rigid is well-defined for the graph $\Gamma_{n,q}(\K)$, since, as said before, we consider this graph to be embedded in $\PG(n+1,q)$.

From now on, we fix an isomorphism $\alpha$ between the graphs $\Gamma_{n,q}(\K)$ and $\Gamma_{n,q}(\K')$ that preserves $\P$, and if a
subspace is called rigid, we mean $\a$-rigid. It follows from the
definition of the graph $\Gamma_{n,q}(\K)$ that every point of
$\mathcal{\K}$ is rigid.

It is clear that if an isomorphism $\beta$ between $\Gamma_{n,q}(\K)$ and $\Gamma_{n,q}(\K')$ (or $T_n^*(\K)$ and $T_n^*(\K')$) is induced by a collineation, necessarily, every point of $H_\infty$ is $\beta$-rigid. Our way of approaching question (Q) is to find conditions on the point sets $\K$, $\K'$ that force every point of $H_\infty$ to be rigid (for all isomorphisms between $\Gamma_{n,q}(\K)$ and $\Gamma_{n,q}(\K')$). Finally, these conditions will enable us to prove that every isomorphism between $\Gamma_{n,q}(\K)$ and $\Gamma_{n,q}(\K')$ (or $T_n^*(\K)$ and $T_n^*(\K')$) is induced by a collineation.

We start of with an easy lemma on rigid subspaces.

\begin{lemma}\label{intersectie} If $\pi_1$ and $\pi_2$ are rigid
subspaces meeting in at least one point, then $\pi_1\cap \pi_2$ is a rigid
subspace.
\end{lemma}
\begin{proof} Let $R$ be an affine point, let
$\dim(\pi_i)=k_i$ and $\dim(\pi_1\cap \pi_2)=m\geq 0$. Since $\pi_i$ is
rigid, the affine points of $\langle R,\pi_i\rangle$ are mapped onto the
affine points of a $(k_i+1)$-dimensional space $\mu_i$. As $\mu_1$ and
$\mu_2$ are subspaces of $\PG(n+1,q)$, they intersect in a subspace of
$\PG(n+1,q)$. Moreover, as this subspace contains exactly the images
under $\a$ of the points that are contained in $\langle
R,\pi_1\rangle\cap\langle R, \pi_2\rangle$, it has dimension $m+1$.
This implies that $\pi_1\cap \pi_2$ is rigid.
\end{proof}

In the previous lemma, we proved that the intersection of rigid subspaces is a rigid space. The next step towards the proof of our main theorem is to show that the span of rigid subspaces is again a rigid subspace. Unfortunately, we need to impose certain restrictions to be able to show this result.

Recall that the points of $\K$ are rigid. In the next lemma, we will give a condition that ensures that the span of two points of $\K$ is a rigid line. This result will be generalised in Lemma \ref{opspanning} to arbitrary rigid points and spaces of arbitrary dimension.

\begin{definition} For a line $L$ of $\PG(n+1,q)$ not in $H_\infty$, we define $\infty(L)$ to be the point $L\cap H_\infty$.\end{definition}

\begin{lemma} \label{help}Let $q>2$. Suppose that no plane of $H_\infty$ intersects $\K$ or $\K'$ in two intersecting lines or in two intersecting lines minus their intersection point. Let $P_1$ and $P_2$ be two points of $\K$, then $P_1P_2$ is rigid.
 \end{lemma}
\begin{proof}
Both $\alpha$ as $\alpha^{-1}$ define an isomorphism. From our proof it will follow that if $\K$ has the property that no plane of $H_\infty$ intersects it in two intersecting lines or in two intersecting lines minus their intersection point, the set $\K'$ has this property as well, and conversely. Hence it does not matter whether $\K$ or $\K'$ has the property; we assume that no plane of $H_\infty$ intersects $\K'$ in two intersecting lines or in two intersecting lines minus their intersection point.

First, suppose the line $N$ of $H_\infty$ containing the points $P_1, P_2$ of $\K$, contains an extra point $P_3\in\K$. Let $\pi$ be a plane through $P_1P_2$, not in $H_\infty$, let $L_i$, $i=1,\ldots,q$, be the $q$ lines in $\pi$ through $P_1$, different from $P_1 P_2$, and let $M_j$, $j=1,\ldots,q$, be the $q$ lines through $P_2$ in $\pi$, different from $P_1 P_2$. It is clear that the distance between $L_1$ and $L_i$, $i=2,\ldots,q$ in $\Gamma_{n,q}(\K)$ is $4$.

Suppose that $\infty(\a(L_1))$ is
different from $\infty(\a(L_i))$ for some $2\leq i\leq q$. If for some $j\neq k$, $\a(M_j)$ and $\a(M_k)$ meet
$H_\infty$ in the same point, then $\a(L_1)$ and $\a(L_i)$ are lines in the
same plane, hence, meeting in an affine point. This would mean that
$\a(L_1)$ and $\a(L_i)$ are at distance $2$ in $\Gamma_{n,q}(\K')$; a
contradiction since $\a$ is an isomorphism.

This implies that if $\infty(\a(L_1))$ is different from $\infty(\a(L_i))$ for some $i$, then the points $\infty(\a(M_j))$, $1\leq j\leq q$, are distinct and hence all the points $\infty(\a(L_i))$, $1\leq i\leq q$, are also mutually different. Moreover, it is  impossible that $\infty(\a(L_i))=\infty(\a(M_j))$ for some $1 \leq i, j\leq q$, since $L_i$ and $M_j$ are at distance $2$ in $\Gamma_{n,q}(\K)$ and $\a$ is an isomorphism. Using that $\a(L_i)$ meets $\a(M_j)$ in a point for all $1\leq i,j\leq q$, we see that the sets $\a(L_1),\ldots,\a(L_q)$ and $\a(M_1),\ldots,\a(M_q)$ are each contained in a regulus of a hyperbolic quadric $Q^+(3,q)$. However, a line through $P_3$ is mapped to a line that contains $q$ points of this quadric, but is not contained in a regulus of this quadric, a contradiction, hence $P_1 P_2$ is rigid.

Now suppose that the line $N$ does not contain an extra point of $\K$, then we now know that either $P_1P_2$ is rigid, or the lines $\a(L_1),\ldots,\a(L_q)$ and $\a(M_1),\ldots,\a(M_q)$ are each contained in a regulus of a hyperbolic quadric ${Q}^+(3,q)$. This quadric meets $H_\infty$ in a plane $\pi$ containing two lines $N_1,N_2$ (the remaining lines in each of these reguli) and the $2q$ points of $N_1$ and $N_2$, different from their intersection point, are necessarily points of $\K'$.

By our assumption, this is either not possible or there exists a point $P \in \K'$ on $\pi$, not on the lines $N_1$ and $N_2$. Consider a line $T$ through $P$ intersecting $N_1$ in a point $S_1$ and intersecting the line $N_2$ in a point $S_2$, for which $S_1 \neq S_2$. Consider the plane $\eta$ through $T$, different from $\pi$, intersecting the quadric $Q^{+}(3,q)$ in two lines, one line through $S_1$ and the other through $S_2$. By the first part, $T$ is rigid for any isomorphism from $T_n^*(\K')$ to $T_n^*(\K)$, hence the plane $\eta$ is mapped by $\alpha^{-1}$ to the plane $\pi$, so the lines through $S_1$, $S_2$ respectively, are mapped to the lines through $P_1$, $P_2$ respectively. This is a contradiction since the lines of the reguli of the quadric are already mapped by $\alpha^{-1}$ to the lines through $P_1$ and $P_2$ in $\pi$. It follows that $P_1P_2$ is rigid.
\end{proof}

\begin{remark}
For $q=2$, an element of $\L$ contains only two points of $\P$. If $\K=H_\infty$, clearly a permutation of $\P$ will always induce an automorphism of $T_n^*(\K)$. Moreover, consider the linear representation $T_2^*(\K)$ where $q=2$; we checked by computer that whatever point set $\K$ you take in $H_\infty=\PG(2,2)$, the full automorphism group of $T_2^*(\K)$ will always be larger than the automorphism group induced by collineations of $\PG(3,2)$.
\end{remark}

From now on, in the remaining of this section, we {\bf assume} that for all point sets $\K$ of $H_\infty$ there is no plane of $H_\infty$ intersecting $\K$ only in two intersecting lines, or in two intersecting lines minus their intersection point. Sometimes we will explicitly refer to this as the set $\K$ having {\em \bf Property (\texttt{*})}. We also assume that the point set $\K$ spans $H_\infty$.

\begin{lemma}\label{crucial} We can define a mapping $\tilde{\a}$ on the set of rigid points by putting $\tilde{\a}(Q)=\infty(\a(L))$ where
$Q$ is a rigid point and $L$ is any line for which $\infty(L)=Q$.
\end{lemma}
\begin{proof}
We have to show that this mapping is well-defined, hence,
if $Q$ is a rigid point and $L_1$ and $L_2$ are two lines through $Q$,
not contained in $H_\infty$, then we will show that $\infty(\a(L_1))=\infty(\a(L_2))$.

We first show that the mapping is well-defined for points of $\K$. Consider a line $L_1$ through $P_1$ contained in a $k$-space through a $(k-1)$-space $\langle P_1,\ldots,P_k\rangle$, for points $P_i\in \K$. Consider the case $k=2$, then by Lemma \ref{help}, we know that $P_1P_2$ is rigid. Let $\pi$ be the plane $\langle L_1, P_2 \rangle$, this plane intersects $H_\infty$ in the line $P_1 P_2$. The lines of $\pi$ through $P_1$ partition the affine points of $\pi$, and $\a$ is an isomorphism mapping lines of $\pi$ through $P_1$ onto lines, hence $\infty(\a(L))=\infty(\a(L_1))$ is the same point for all lines $L$ in $\pi$ through the point $P_1$, different from $P_1 P_2$.

Now we proceed by induction. Suppose that for every line $L \in \L$ through
$P_1$, contained in a $k$-space through a $(k-1)$-space $\langle P_1,\ldots,P_k\rangle$, $P_i\in \K$,
$k\leq n$, we have $\infty(\a(L))=\infty(\a(L_1))$. Consider a point $P_{k+1} \in \K$ not in $\langle P_1, \ldots, P_k\rangle$, and let $M \in \L$ be any line
through $P_1$, contained in a $(k+1)$-space through the $k$-space $\langle P_1,\ldots,P_{k+1}\rangle$. The
plane $\langle M,P_{k+1}\rangle$ meets the $(k-1)$-space $\langle P_1,\ldots,P_k\rangle$
in a line $N$ for which $\infty(N)=P_1$. By the induction hypothesis, $\infty(\a(N))
=\infty(\a(L_1))$. Moreover, from the case $k=2$, we know that for the
line $M$, it holds that $\infty(\a(M))=\infty(\a(N))$.
Hence, using that the points of $\K$ span $H_\infty$, we have shown
that the theorem is valid for all points of $\K$.

Now suppose that $Q$ is a rigid point not contained in $\K$. Let $L$ be
a line through $Q$. Consider a point $P_i \in \K$ and the plane $\pi = \langle L,P_i\rangle$; let $M$ be a line through $Q$ in $\pi$ different from $L$. Let $N_j$, $j=1,2$, be two lines
in $\pi$, going through the point $P_i$. We have shown in the previous
part that $\infty(\a(N_1))=\infty(\a(N_2))$.  It
follows that $\a(L)$ and $\a(M)$ lie in a plane, hence, they meet in a
point. Since $\a$ is an automorphism and $L$ and $M$ only meet in a
point at infinity, $\a(L)$ and $\a(M)$ cannot meet in an affine point,
hence $\infty(\a(L))=\infty(\a(M))$. The lemma now follows by induction with the same argument as used above
for points of $\K$.
\end{proof}

The previous lemma shows that, an isomorphism $\a$ between $\Gamma_{n,q}(\K)$ and $\Gamma_{n,q}(\K')$ preserving points can be extended to the rigid points $Q \in H_\infty$ by putting $\a(Q):=\tilde{\a}(Q)$.

\begin{lemma} \label{opspanning}If $P_1,\ldots,P_{k+1}$ are rigid
points, then $\langle P_1,\ldots,P_{k+1}\rangle$ is a rigid space.
\end{lemma}
\begin{proof} We proceed by induction on the dimension $k$ of the rigid
space.
Let $k=1$, and let $\pi$ be a plane meeting $H_\infty$ in the
line $P_1P_2$, and let $R$ be an affine point of $\pi$. By Lemma \ref{crucial}, since
$P_1$ and $P_2$ are rigid, the points on the
line $RP_i$, $i=1,2$, are mapped by $\a$ to the points on the line $\langle \a(R), \a(P_i) \rangle$. Let $S\neq
R$ be a point of $\pi$, not on $P_1P_2$, $R P_1$ or $R P_2$ and let $T_1$ (resp. $T_2$) be
the intersection point of the line $SP_1$ (resp. $SP_2$) with $RP_2$
(resp. $RP_1$), then $\a(T_1)$ lies on $\langle \a(R),\a(P_2)\rangle$
and $\a(T_2)$ lies on $\langle \a(R),\a(P_1)\rangle$. It follows from
Lemma \ref{crucial} that $\a(S)$ lies on $\langle \a(T_1),
\a(P_1)\rangle$ and $\langle \a(T_2), \a(P_2)\rangle$, hence, $\a(S)$
is contained in the plane $\langle \a(R),\a(P_1),\a(P_2)\rangle$. It follows that
$P_1P_2$ is rigid.

Now suppose that $\pi:=\langle P_1,\ldots,P_{t-1}\rangle$, with all $P_i$ rigid, is a rigid space of dimension $k\leq n-1$. Let $P_t$ be a
rigid point, different from $P_1, \ldots, P_{t-1}$. If $P_t \in \pi$, the space
$\pi=\langle \pi, P_t \rangle$ is rigid. Suppose $P_t \notin \pi$, then $\mu:=\langle \pi, P_t\rangle$ is a space of dimension $k+1$. Let
$R$ be an affine point of $\PG(n+1,q)$, then, since $\pi$ is rigid and by
Lemma \ref{crucial}, we have that every affine point of $\langle
R,\pi\rangle$ is mapped by $\a$ to an affine point of $\langle
\a(R),\a(P_1),\ldots,\a(P_{t-1})\rangle$. Let $S$ be a point of
$\langle R,\mu\rangle$, not in $\langle R,\pi\rangle$. If $S$ lies on
$RP_t$, then $\a(S)$ is certainly contained in $\langle
\a(R),\a(P_1),\ldots,\a(P_{t-1}),\a(P_t)\rangle$, hence, we may
suppose that $S$ does not lie on $RP_t$. The line
$SP_t$ meets $\langle R,\pi\rangle$ in a point $T$. As
$\a(T)$ lies in  $\langle
\a(R),\a(P_1),\ldots,\a(P_{t-1})\rangle$ and $\a(S)$ lies on the
line through $\a(P_t)$ and $\a(T)$ by Lemma \ref{crucial}, $\a(S)$ lies
in $\langle \a(R),\a(P_1),\ldots,\a(P_{t-1}),\a(P_t)\rangle$. This
proves our lemma.
\end{proof}

\begin{theorem}Let $q>2$. Let $\K$ and $\K'$ denote point sets having Property (\texttt{*}) in $H_\infty=\PG(n,q)$ such that the closure $\overline{\K}$ is equal to $H_\infty$, and such that, if $\vert \K\vert=q$ every point of $H_\infty$ lies on at least one tangent line to $\K$. Let $\alpha$ be an isomorphism between $\Gamma_{n,q}(\K)$ and $\Gamma_{n,q}(\K')$. Then $\alpha$ is induced by an element of stabiliser $\PGammaL(n+2,q)_{H_\infty}$ mapping $\K$ onto $\K'$.
\end{theorem}
\begin{proof} By Corollary \ref{cor22}, every isomorphism between $\Gamma_{n,q}(\K)$ and $\Gamma_{n,q}(\K')$ maps points onto points. Hence, the concept `rigid' is well-defined. It follows from the construction of
the closure of a set and from Lemmas \ref{intersectie} and \ref{opspanning}, that all points in $\overline{\K}=H_\infty$ are rigid. Hence $\alpha$ maps collinear points of a line, not in $H_\infty$, onto collinear points.

As announced before, we abuse notation and use $\a$ for the extension of $\a$ to all points of $\PG(n+1,q)$. Let $P,Q,R$ be three points on a line $L$ contained in $H_\infty$ and let $S$ be a point, not contained in $H_\infty$. Since $L$ is a rigid line, $\a$ maps the points of $\langle S,L\rangle$ onto points of a plane
containing $\a(P), \a(Q)$, and $\a(R)$ at infinity. This implies that $\a$ also maps collinear points of $H_\infty$ to collinear points of $H_\infty$, hence, is a collineation that stabilises ${H_\infty}$.
\end{proof}

If $\overline{\K}$ is $H_\infty$, clearly $\K$ spans $H_\infty$ and hence, $(\PGammaL(n+2,q)_{H_\infty})_\K=\PGammaL(n+2,q)_\K$. Taking this into account, the previous theorem has the following corollary.

\begin{corollary} \label{hoofd1}Let $q>2$. Let $\K$ denote a point set having Property (\texttt{*}) in $H_\infty=\PG(n,q)$ such that the closure $\overline{\K}$ is equal to $H_\infty$, and such that if $\vert \K\vert=q$, every point of $H_\infty$ lies on at least one tangent line to $\K$. Then $\Aut(\Gamma_{n,q}(\K)) \cong \PGammaL(n+2,q)_\K$.
\end{corollary}

Since an isomorphism between incidence structures $T_n^*(\K)$ and $T_n^*(\K')$  {\em by definition} maps points to points, the assumption that every point of $H_\infty$ lies on a tangent line to $\K$ and $\K'$ (showing that points and lines cannot be mapped onto each other) can be dropped, and we have the following corollary.

\begin{corollary}\label{hoofd} Let $q>2$. Let $\K$ and $\K'$ denote point sets having Property (\texttt{*}) in $H_\infty=\PG(n,q)$ such that the closure $\overline{\K}$ is equal to $H_\infty$ and let $\alpha$ be an isomorphism between $T_n^*(\K)$ and $T_n^*(\K')$. Then $\alpha$ is induced by an element of $\PGammaL(n+2,q)_{H_\infty}$ mapping $\K$ to $\K'$.
\end{corollary}

In the next section, we will give some examples of non-geometric automorphisms of $T_n^*(\K)$ for point sets $\K$ not satisfying the conditions of Corollary \ref{hoofd}.

\section{Non-geometric automorphisms}\label{sec4}

In Corollary \ref{hoofd}, it is assumed that $\K$ satisfies Property (\texttt{*}) and that the closure of $\K$ is $H_\infty$. It turns out that, if one of these conditions is not satisfied, there exist counterexamples to this corollary. We give explicit constructions of such counterexamples and provide computer results that give more information on $\Aut(T_n^*(\K))$ for small $n$ and $q$. All the mentioned computer results were obtained with GAP \cite{GAP}.

\subsection{Point sets not satisfying Property (\texttt{*})}
Recall that Property (\texttt{*}) states that $\K$ is a point set such that there is no plane of $H_\infty$ intersecting $\K$ only in two intersecting lines, or in two intersecting lines minus their intersection point.
Let $\K$ be the point set of two intersecting lines in $\PG(2,q)$, we will show that there exist non-geometric automorphisms of $T_2^*(\K)$, and construct the automorphism group $\Aut(T_2^*(\K))$.

From the findings in the proof of Lemma \ref{help} and the following lemmas in Section \ref{sec3} we deduce that an automorphism $\phi$ of $T_2^*(\K)$ in this case is either geometric, and thus induced by an element of $\PGammaL(4,q)_\K$, or is some non-geometric mapping of $T_2^*(\K)$. The lines of $\L$ in planes intersecting $\K$ in exactly two points are either mapped to lines in planes or to lines of hyperbolic quadrics $Q^+(3,q)$. It is easy to see that, if the lines of $\L$ in some plane that intersects $\K$ in exactly two points are mapped to the lines of a hyperbolic quadric $Q^{+}(3,q)$, then this is true for all planes that intersect $\K$ in exactly two points, just by looking at the intersection of two planes and the intersection of a plane and a hyperbolic quadric. We will construct such a mapping and show that if $\psi$ and $\psi'$ are non-geometric, then $\psi' = \chi_1 \psi \chi_2$ with $\chi_i \in \PGammaL(4,q)_\K$.

Without loss of generality, let $H_\infty$ be the plane of $\PG(3,q)$ with equation $X_0=0$, and let the set $\K$ consist of the points of two intersecting lines $L_1: X_0=X_1=0$ and $L_2: X_0=X_2=0$.
\begin{theorem} For the set of affine points $\P$ of $T_2^*(\K)$, the mapping \[\phi : \P \rightarrow \P : (1,x,y,z) \mapsto (1,x,y,z+xy)\] induces a non-geometric automorphism of $T_2^*(\K)$.
\end{theorem}
\begin{proof}

The map $\phi$ is clearly a bijection.

We will describe the action of $\phi$ on all lines not in $H_\infty$. Lines through $(0,0,0,1)$ not in $H_\infty$ are stabilised by $\phi$. A line $M$ of $T_2^*(\K)$ through $(0,1,0,u)$, $u \in \F_q$, such that $\langle M, L_2\rangle $ is the plane with equation $yX_0 -X_2=0$, with $y \in \F_q$, is mapped by $\phi$ to a line of $T_2^*(\K)$ through $(0,1,0,u+y)$.
A line $N$ of $T_2^*(\K)$ through $(0,0,1,u')$, $u' \in \F_q$,  such that $\langle N, L_1\rangle$ is the plane with equation $xX_0 -X_1=0$, with $x \in \F_q$, is mapped by $\phi$ to a line of $T_2^*(\K)$ through $(0,0,1,u'+x)$.
The affine points of a line through $(0,1,v,w)$, $v,w \in \F_q$, $v \neq 0$, not in $H_\infty$, are mapped by $\phi$ to the affine points of an irreducible conic containing the point $(0,0,0,1)$. More specifically, the affine points of the line $\langle (1,x,y,z),(0,1,v,w)\rangle$, $v \neq 0$, are mapped to the points of the conic with equation $(z-wx)X_0^2+vX_1^2+(w+y-vx)X_0X_1-X_0X_3=0$, different from $(0,0,0,1)$. This conic is contained in the plane $X_2 = (y-vx)X_0+vX_1$. The mapping $\phi$ induces an automorphism of $T_2^*(\K)$, but is clearly not induced by a collineation. \end{proof}

\begin{definition}
If a group $G$ has a normal subgroup $N$ and the quotient $G/N$ is isomorphic
to some group $H$, we say that $G$ is an {\em extension} of $N$ by $H$. This is written as $G = N.H$.

An extension $G = N.H$ which is a semidirect product is also called a {\em split
extension} and is denoted by $G=N \rtimes H$. This means that one can find a subgroup $\overline{H}\cong H$ in $G$ such that $G = N\overline{H}$ and
$N \cap \overline{H} = \{e_G \}$.
\end{definition}

Consider the group $S$ of non-geometric automorphisms of $T_2^*(\K)$  induced by $\{ \phi_m:\P\rightarrow\P:(1,x,y,z)\mapsto(1,x,y,z+mxy) \mid m \in \F_q \}$. It is clear that $S$ is isomorphic to $(\F_q,+)$.
\begin{theorem}
The group $\langle \PGammaL(4,q)_\K, \phi \rangle = S \rtimes \PGammaL(4,q)_\K$ is the full automorphism group of $T_2^*(\K)$ and is $q$ times larger than the geometric automorphism group $\PGammaL(4,q)_\K$.
\end{theorem}
\begin{proof}
Consider a non-geometric automorphism $\psi$ of $T_2^*(\K)$. We look at its action on the planes through the line $N$ with equation $X_0=X_3=0$.
One can check that the pencil of planes $\{ aX_0+X_3=0 \mid a \in \F_q\}$ is mapped by $\psi$ to the pencil of hyperbolic quadrics $\{ X_0(a'X_0+X_3)+mX_1X_2=0 \mid a' \in \F_q\}$ for some non-zero $m \in \F_q$. By multiplying with a well-chosen element of $\PGammaL(4,q)_\K$, we may say that the pencil of planes $\{ aX_0+X_3=0 \mid a \in \F_q\}$ is mapped by $\psi$ to the pencil of hyperbolic  quadrics $\{ X_0(a'X_0+X_3)+X_1X_2=0 \mid a' \in \F_q\}$. Also, by using a well-chosen collineation of $\PGammaL(4,q)_\K$ which switches the lines $L_1$ and $L_2$ if necessary, the pencils of planes $\{bX_0+X_1=0 \mid b \in \F_q\}$ and $\{cX_0+X_2=0 \mid c \in \F_q\}$ are both fixed by $\psi$.

The mapping $\phi$ maps every plane $aX_0+X_3=0$ to the hyperbolic quadric $ X_0(aX_0+X_3)+X_1X_2=0$ and fixes all the planes of the form $bX_0+X_1=0$ and $cX_0+X_1=0$. We now consider the mapping $\phi^{-1}\psi$, this map sends planes through $N$ to planes through a line of $H_\infty$,  hence, as seen before, sends all planes to planes. Thus, $\phi^{-1}\psi$ is induced by some collineation of $\PGammaL(4,q)_\K$. We see that $\psi=\chi_1 \phi \chi_2$ with $\chi_i \in \PGammaL(4,q)_\K$ and hence $\langle \PGammaL(4,q)_\K , \phi \rangle$ is the full automorphism group of $T_2^*(\K)$.

Since the group $S$ has order $q$ and is a normal subgroup of $\langle \PGammaL(4,q)_\K , \phi \rangle$ such that $S \cap \PGammaL(4,q)_\K$ is trivial, the theorem follows.
\end{proof}

For $q=3,4$, we take $\K$ to be the point set of two intersecting planes in $H_\infty=\PG(3,q)$. We see that the group $\Aut(T_3^*(\K))$ is $q^2$ times larger than $\PGammaL(5,q)_\K$. Hence, also in this case there exist automorphisms of $T_3^*(\K)$ which are not induced by a collineation of the ambient projective space.

\begin{remark}\label{rem3}  There are point sets $\K$ that do not satisfy Property (\texttt{*}) but do have the property that $\Aut(T_3^*(\K))$ consists entirely of geometric automorphisms. E.g. let $\K$ be the point set of three lines $L_1, L_2, L_3$ in $H_\infty=\PG(3,q)$ such that $L_1 \cap L_2=\emptyset$ and $L_3$ intersects $L_1$ and $L_2$, then, by going through the details of the proofs in the previous section, it is not too hard to check that $\Aut(T_3^*(\K))=\PGammaL(5,q)_\K$.\end{remark}

\subsection{Point sets $\K$ with closure different from $H_\infty$}

\subsubsection{The span of $\K$ is a proper subspace of $H_\infty$}
In this case, it is fairly easy to construct non-geometric automorphisms, as seen in the following theorem.

\begin{theorem} Let $\K$ be a set of points in $H_\infty=\PG(n,q)$, $q\geq 4$, such that $dim\langle \K\rangle\leq n-1$, then there exist non-geometric automorphisms of $T_n^*(\K)$.
\end{theorem}
\begin{proof} Let $\pi$ be the subspace of dimension $d$, spanned by $\K$ in $H_\infty$ and let $\mu_1$ and $\mu_2$ be $(d+1)$-dimensional spaces in $\PG(n+1,q)$, meeting $H_\infty$ exactly in $\pi$. Let $Q$ be a point of $H_\infty$, not contained in $\pi$. Let $\phi$ be the mapping from points $P$ in $\P$ to $\P$ defined as follows. When $P\in \mu_1$, then $\phi(P)=\langle Q,P\rangle \cap \mu_2$, when $P\in \mu_2$, then $\phi(P)=\langle Q,P\rangle\cap \mu_1$,  in all other cases $\phi(P)=P$. It is clear that $\phi$ is a bijection and that a line through a point $R$ of $\pi$ is mapped onto a line through $R$, thus $\phi$ is an automorphism of $T_n^*(\K)$. Let $M$ be a line through a point $S\neq Q$ of $H_\infty\setminus \pi$, then we see  that $q-2$ points of $M$ are fixed, but the intersection points of $M$ with $\mu_1$ and $\mu_2$ are mapped onto points that are not on $M$. Hence, $\phi$ is not induced by a collineation of $\PG(n+1,q)$.
\end{proof}

\subsubsection{The span of $\K$ is $H_\infty$ but the closure of $\K$ is not}
We will construct a non-geometric automorphism of $T_2^*(\K)$, where $\K$ is a Baer subplane of $\PG(2,q^2)$. It is worth noticing that in this case, $T_2^*(\K)$ is a semipartial geometry, and that the general belief was that for every $T_m^*(\K)$ that is a semipartial geometry, every automorphism is geometric (see e.g. \cite[Remark 7.3.13]{stefaan2}).

In this construction, we use the representation of $\PG(3,q^2)$ in $\PG(6,q)$ of Barlotti-Cofman \cite{Barlotti}. The points of the hyperplane $H_\infty$ in $\PG(3,q^2)$ are represented as lines of a Desarguesian spread $\D$ in $J_\infty=\PG(5,q)$. The affine points of $\PG(3,q^2)$ with respect to $H_\infty$ can be identified with the affine points of a projective space $\PG(6,q)$, with respect to the hyperplane $J_\infty$. The lines of $\PG(3,q^2)$ not in $H_\infty$ correspond to the planes of $\PG(6,q)$, meeting $J_\infty$ in a line of $\D$.

We start with a lemma.
\begin{lemma} \label{hulp}For $q>2$, let $\D$ be a Desarguesian spread in $\PG(5,q)$ representing $\PG(2,q^2)$ and let $B$ be the set of lines of $\D$ corresponding to a Baer subplane of $\PG(2,q^2)$. Then there exist an element $\psi$ of $\PGammaL(6,q)$ and a line $L$ of $\D$ such that $\psi$ stabilises $B$ setwise and such that $\psi(L)$ is a line meeting $L$ in a point.
\end{lemma}
\begin{proof} Let $\R$ be a regulus contained in $B$ and let $\D'$ be the Desarguesian spread consisting of elements of $\D$ in the $3$-space $\Pi$ spanned by the elements of $\R$. The stabiliser of $\R$ acts transitively on the lines of $\Pi$ that do not intersect a line of $\R$: the orbits of the stabiliser of $\R$ on lines of $\Pi$ not intersecting $\R$ are in one-to-one correspondence with isotopy classes of semifields of order $q^2$ with center containing $q$ (see \cite{cardinali}). By a classical result of Dickson, every such semifield is a field \cite{Dickson}. This implies that there is only one such isotopy class, so for $q>2$, if $L$ is a line of $\D'$, not in $\R$, then there exists an element $\phi$ of $\PGammaL(4,q)$ stabilizing $\R$ and mapping $L$ to a line $L'$ intersecting $L$ in a point. If $q=2$, there are only two lines disjoint from $\R$ and they both belong to $\D'$, so such line $L'$ does not exist.
Every collineation of $\PG(1,q^2)$ induces a collineation of $\PG(3,q)$ stabilising the spread $\D'$, hence, as there exists a collineation of $\PG(1,q^2)$ stabilising the subline corresponding to $\R$ and fixing one point of this subline, this implies that there is a collineation $\phi$ of $\Pi$ stabilising $\R$, fixing one element $R_1$ of $\R$ elementwise and mapping an element $M$ of $\D'$ onto a line meeting $M$ in a point.

Let $L_1$ and $L_2$ be two skew lines in $B$, such that $\langle L_1,L_2\rangle$ meets $\Pi$ exactly in the line $R_1$. The set $\D'\cup \{L_1,L_2\}$ extends to a unique Desarguesian spread of $\PG(5,q)$, which is necessarily equal to $\D$.

Without loss of generality, we may assume that $R_1$ is the line through the points $(1,0,0,0,0,0)$ and $(0,1,0,0,0,0)$, while $L_1$ is the line $\langle(0,0,0,0,1,0),(0,0,0,0,0,1)\rangle$ and $L_2$ is the line $\langle(1,0,0,0,1,0),(0,1,0,0,0,1)\rangle$. If the collineation $\phi$ is given by the matrix $A=(a_{ij})$, $0\leq i,j\leq 3$, then let $C=(c_{ij})$, $0\leq i,j\leq 5$,
be the matrix with $c_{ij}=a_{ij}$ for $0\leq i,j\leq 3$, $c_{44}=a_{00}$,  $c_{45}=a_{01}$,  $c_{54}=a_{10}$,  $c_{55}=a_{11}$ and $0$ elsewhere. Now $C$, together with $\theta \in \Aut(\F_q)$ corresponding to $\phi$, defines an element $\psi$ of $\PGammaL(6,q)$, which stabilises $\R$ and fixes $R_1,L_1$ and $L_2$; these elements of $\D$ define a unique Baer subplane of $\PG(2,q^2)$.

Let $L_3$ be a line of $B$, not contained in $\Pi$ or $\langle L_1,L_2\rangle$, then $L_3=\langle L_1,R_2\rangle\cap \langle L_2,R_3\rangle$ for some $R_2,R_3$ in $\R$. It follows that $\psi(L_3) =\langle \psi(L_1),\psi(R_2)\rangle\cap \langle \psi(L_2),\psi(R_3)\rangle$ =$\langle L_1,R_4\rangle\cap \langle L_2,R_5\rangle$ for some $R_4,R_5$ in $\R$, which is clearly contained in the set $B$. A line $L_4$, contained in $\langle L_1,L_2\rangle$ can now, in the same way, be written as the intersection of $3$-spaces spanned by elements of $B$, which are in the previous part showed to be mapped by $\psi$ onto elements of $B$. Hence, also the line $L_4$ is mapped onto a line of $B$ by $\psi$ and the statement follows.
\end{proof}

In the following theorem, we will show that the collineation constructed in the previous lemma corresponds to a permutation of the affine points of $\PG(3,q^2)$ that is not a collineation.

\begin{theorem} For $q>2$, let $\K$ be a Baer subplane in $\PG(2,q^2)$, then there exist non-geometric automorphisms of the semipartial geometry $T_2^*(\K)$.
\end{theorem}
\begin{proof} Let $\psi$ be the collineation of $\PG(5,q)$ stabilising the line set $B$ of $J_\infty=\PG(5,q)$ setwise and such that $\psi(L)\cap L\neq \emptyset$ for some $L\in \D$, as found in Lemma \ref{hulp}. The element $\psi$ extends to a collineation $\phi$ of $\PG(6,q)$ preserving affine points (with respect to $J_\infty$). Let $\chi$ be the permutation of affine points of $\PG(3,q^2)$, corresponding to $\phi$ and let $P,Q,R$ be three collinear affine points, lying on a line $M$ of $\PG(3,q^2)$, such that $M$ meets $H_\infty$ in a point of $\K$. The points $P,Q,R$ correspond to $3$ affine points of $\PG(6,q)$, contained in a plane $\mu$ through a line of $B$. Now $\phi$ maps $\mu$ onto a plane through an element of $B$, which corresponds to a line $N$ of $\PG(3,q^2)$ through a point of $\K$. It follows that $\chi(P),\chi(Q),\chi(R)$ belong to $N$, and hence, that $\chi$ defines an automorphism of $T_2^*(\K)$.

Let $S$ be the point of $H_\infty$ corresponding to $L$ and let $M'$ be a line in $\PG(3,q^2)$ meeting $H_\infty$ in $S$. The affine points of $M'$ correspond to the affine points of a plane $\mu'$ through $L$. These are mapped by $\phi$ onto the affine points of a plane $\nu$, where $\nu$ meets $q+1$ different lines of the spread $\D$. Hence, the affine points of $M'$ are mapped by $\chi$ onto the points of a Baer subplane. This shows that $\chi$ is not induced by a collineation of $\PG(3,q^2)$.
\end{proof}

It follows from the proof of the previous theorem that $\PGammaL(7,q)_B\leq \Aut(T_2^*(\K))$, where $\K$ is a Baer subplane in $\PG(2,q^2)$ and $B$ denotes the set of lines corresponding to $\K$ of the Desarguesian spread $\D$, corresponding to $\PG(2,q^2)$. It is not hard to calculate that $\PGammaL(7,q)_B$ is $q(q-1)/2$ times larger than the group of geometric automorphisms $\PGammaL(3,q^2)_\K$. It would be interesting to know whether $\PGammaL(7,q)_B$ is the full automorphism group of $\Aut(T_2^*(\K))$ when $\K$ is a Baer subplane. Unfortunately, with the methods developed in Section \ref{sec3} we cannot prove this result. For $q=2,3,4$, we were able to confirm this result by computer (by comparing orders).

Finally, if we consider the point set $\K=\PG(2,2)$ embedded in $H_\infty=\PG(2,8)$, then the group $\Aut(T_2^*(\K))$ turns out to be eight times larger than $\PGammaL(4,8)_\K$. If $\K=\PG(3,3)$ in $\PG(3,9)$, then we see that the group $\Aut(T_3^*(\K))$ is three times larger than the group of geometric automorphisms.

\section{The description of the automorphism group}\label{sec5}

Since an element of $(\PGammaL(n+2,q)_{H_\infty})_\K$ induces a geometric automorphism, it defines an element of $\Aut(\Gamma_{n,q}(\K))$ (or $\Aut(T_n^*(\K))$). In the previous section, we have shown that, under certain conditions, $(\PGammaL(n+2,q)_{H_\infty})_\K$ is the full automorphism group of $\Gamma_{n,q}(\K)$ (or $T_n^*(\K)$). It is our goal to provide a more explicit description of  $(\PGammaL(n+2,q)_{H_\infty})_\K$ in this section.

If the set of elements of $\PGammaL(n+2,q)$ fixing all points of
the hyperplane $H_\infty$ is written as $\Persp(H_\infty)$, then
$\Persp(H_\infty)$ consists of all elations and homologies with axis
$H_\infty$.

\begin{lemma} \label{ext} The group $(\PGammaL(n+2,q)_{H_\infty})_\K$ is an extension of $\Persp(H_\infty)$ by $\PGammaL(n+1,q)_\K$ and $(\PGL(n+2,q)_{H_\infty})_\K$ is an extension of $\Persp(H_\infty)$ by $\PGL(n+1,q)_\K$.
\end{lemma}
\begin{proof}
The kernel of the action of $(\PGammaL(n+2,q)_{H_\infty})_\K$, $(\PGL(n+2,q)_{H_\infty})_\K$ respectively, on $H_{\infty}$ is clearly
$\Persp(H_\infty)$. The image of the action is isomorphic to
$\PGammaL(n+1,q)_\K$, $\PGL(n+1,q)_\K$ respectively, showing that $(\PGammaL(n+2,q)_{H_\infty})_\K$, $(\PGL(n+2,q)_{H_\infty})_\K$ respectively, is an extension of
$\Persp(H_\infty)$ by $\PGammaL(n+1,q)_\K$, $\PGL(n+1,q)_\K$ respectively.
\end{proof}

\begin{remark}
In general, $\PGammaL(n+2,q)_{H_\infty}$ is an extension of $\Persp(H_\infty)$ by $\PGammaL(n+1,q)$. However, this extension does not necessarily split since $\PGammaL(n+1,q)$ is not necessarily embeddable in $\PGammaL(n+2,q)$. For example, $\PGL(4,4)$ has no subgroup isomorphic to $\PGL(3,4)$. Depending on the choice of $\K$, we can investigate whether $(\PGammaL(n+2,q)_{H_\infty})_\K$ does split over $\Persp(H_\infty)$. To show that this extension splits, we need to embed the group $\PGammaL(n+1,q)_\K$ in $(\PGammaL(n+2,q)_{H_\infty})_\K$, and different groups $\PGammaL(n+1,q)_\K$ may require a different proof. In the next theorem, we give a general condition on $\K$ that is sufficient to show that the extension splits. The condition is not necessary: in \cite{wij}, the theorem is shown to hold when $\K$ is a basis or frame in $H_\infty$, and in \cite{stefaan}, it is shown that the same holds for Buekenhout-Metz unitals. However, our theorem proves the theorem for a lot of different point sets $\K$ at the same time.
\end{remark}
We start with an easy lemma.

\begin{lemma} \label{aut} If there is an element of $\PGammaL(n+1,q)$ mapping $\K$ to a point set $\K'$ that is stabilised under the Frobenius automorphism, then
$\PGammaL(n+1,q)_\K\cong \PGL(n+1,q)_\K \rtimes \Aut(\F_q)$.
\end{lemma}
\begin{proof} Since all automorphisms of $\F_q$ are generated by the Frobenius automorphism, every automorphism of $\F_q$ stabilises $\K'$. From this, if there is an element of $\PGammaL(n+1,q)$ mapping $\K$ to $\K'$, we can also find an element of $\PGL(n+1,q)$ mapping $\K$ to $\K'$. Since $\K'$ is contained in the orbit of $\K$ under $\PGL(n+1,q)$, $\PGL(n+1,q)_\K\cong \PGL(n+1,q)_{\K'}$ and $\PGammaL(n+1,q)_\K\cong \PGammaL(n+1,q)_{\K'}$. Since $\Aut(\F_q)$ stabilises $\K'$, we can restrict the well-known isomorphism $\PGammaL(n+1,q)\cong \PGL(n+1,q) \rtimes \Aut(\F_q)$ to elements of $\PGammaL(n+1,q)_{\K'}$, and the lemma follows.
\end{proof}

\begin{theorem} \label{isom} If the setwise stabiliser $\PGammaL(n+1,q)_\K$, respectively $\PGL(n+1,q)_\K$, of a point set $\K$ spanning $H_\infty=\PG(n,q)$, $q=p^h$, fixes a point of $H_\infty$, then
$\PGammaL(n+2,q)_\K \cong \Persp(H_\infty)\rtimes \PGammaL(n+1,q)_\K$, respectively $\PGL(n+2,q)_\K \cong \Persp(H_\infty)\rtimes \PGL(n+1,q)_\K$.
\end{theorem}
\begin{proof} Since $\K$ spans $H_\infty$, $\PGammaL(n+2,q)_\K$ is contained in $\PGammaL(n+2,q)_{H_\infty}$. By Lemma \ref{ext}, we see that $\PGammaL(n+2,q)_\K$ is an extension of $\Persp(H_\infty)$ by $\PGammaL(n+1,q)_\K$. This extension splits if and only if $\PGammaL(n+1,q)_\K$ can be embedded in $\PGammaL(n+2,q)_{\K}$ in such a way that it intersects trivially with $\Persp(H_\infty)$. By assumption, $\PGammaL(n+1,q)_\K$ fixes a point $P\in H_\infty$.  Suppose that $P$ has coordinates $(0,c_1,c_2,\ldots,c_{n+1})$, where the first non-zero coordinate equals one. This implies that for each $\beta\in \PGammaL(n+1,q)_\K$, there exists a unique $(n+1)\times (n+1)$ matrix $B=(b_{ij})$, $1\leq i,j\leq n+1$, and an element $\theta \in \Aut(\F_q)$ corresponding to $\beta$, such that $(c_1,c_2,\ldots,c_{n+1})^\theta.B=(c_1,c_2,\ldots,c_{n+1})$. Moreover, the obtained semi-linear maps $(B,\theta)$ form a subgroup $G$ of $\Gamma\mathrm{L}(n+1,q)$. Let $A_\beta=(a_{ij})$, $0\leq i,j\leq n+1$, be the $(n+2)\times (n+2)$ matrix with $a_{00}=1$, $a_{i0}=a_{0j}=0$ for $i,j\geq 1$ and $a_{ij}=b_{ij}$ for $1\leq i,j\leq n+1$. It is clear that $(A_\beta, \theta)$ defines an element of $\PGammaL(n+2,q)_{H_\infty}$, corresponding to a collineation $\alpha$ acting in the same way as $\beta$ on $H_\infty$. If $\theta \neq \mathbbm{1}$, then $\alpha$ is not a perspectivity. If $\theta = \mathbbm{1}$, then $\alpha$ fixes every point on the line through $P$ and $(1,0,\ldots,0)$, thus fixes at least two affine points, and hence is not a perspectivity. This implies that we have found a subgroup of $\PGammaL(n+2,q)_{\K}$ isomorphic to $\PGammaL(n+1,q)_\K$ and intersecting $\Persp(H_\infty)$ trivially.

The second claim can be proved in the same way.
\end{proof}

\begin{corollary}
If the setwise stabiliser $\PGL(n+1,q)_\K$ of a point set $\K$ spanning $H_\infty=\PG(n,q)$ fixes a point of $H_\infty$, and $\PGammaL(n+1,q)_\K \cong \PGL(n+1,q)_\K \rtimes \Aut(\F_{q_0})$, for $q_0=p^{h_0}$, $h_0 | h$ or $\PGammaL(n+1,q)_\K \cong \PGL(n+1,q)_\K $, then
$\PGammaL(n+2,q)_\K \cong \Persp(H_\infty)\rtimes \PGammaL(n+1,q)_\K$.
\end{corollary}
\begin{proof}
It is clear that $\Aut(\F_q)$ can be embedded in $\PGammaL(n+2,q)$ by mapping $\theta \in \Aut(\F_q)$ to the semi-linear map $(I,\theta) \in \PGammaL(n+2,q)$ where $I$ is the $(n+2) \times (n+2)$ identity matrix. Since $\Persp(H_\infty)$ intersects $\Aut(\F_{q})$ trivially, the corollary follows.
\end{proof}

Examples of point sets satisfying the conditions of Theorem \ref{isom} are ubiquitous; the case that $\K$ is a $q$-arc in $H_\infty$ is studied in detail in \cite{wij}.

\noindent
Affiliation of the authors:\\

\noindent
Department of Mathematics\\
Vrije Universiteit Brussel\\
Pleinlaan 2\\
1050 Brussel\\

\noindent
\{pcara,srottey,gvdevoor\}\makeatletter @vub.ac.be

\end{document}